\documentclass[11pt]{article}
\usepackage{amsthm,amsmath,amssymb}
\usepackage{graphicx,float,color,fancybox,setspace,hyperref}
\usepackage{pgf,tikz}
\usepackage[affil-it]{authblk}
\usepackage{indentfirst}
\usepackage[section]{placeins}
\usepackage{enumerate}
\usepackage{thmtools}
\usepackage{enumitem}
\usepackage{subcaption}
\usepackage{bm}
\usepackage{mathrsfs}
\usetikzlibrary{positioning, shapes.geometric, arrows.meta, calc, fit,decorations.pathreplacing}

\hypersetup{
	colorlinks,
	linkcolor={red!100!black},
	citecolor={green!60!black},
	urlcolor={blue!60!black}
}
\newtheorem{theorem}{Theorem}
\newtheorem{lemma}{Lemma}

\theoremstyle{definition}
\newtheorem{definition}{Definition}

\newcommand{\Fq}{\mathbb F_q}
\newcommand{\Gr}[2]{\genfrac{[}{]}{0pt}{}{#1}{#2}}
\newcommand{\qbinom}[2]{\genfrac{[}{]}{0pt}{}{#1}{#2}_q}
\newcommand{\qnum}[1]{[#1]_q}
\newcommand{\rank}{\operatorname{rank}}
\newcommand{\Hom}{\operatorname{Hom}}
\newcommand{\co}{\operatorname{co}_2}
\newcommand{\AG}{\operatorname{AG}}
\newcommand{\Tr}{\operatorname{Tr}}
\newcommand{\dir}{\operatorname{dir}}
\newcommand{\Span}{\operatorname{span}}
\newcommand{\aff}{\operatorname{aff}}

\begin{document}
\title{Erd\H{o}s--Ko--Rado theorems in $\ell_2$-norm for three finite spaces}
\author{Qian BAO$^1$, Yaojun CHEN$^{1,}$\footnote{Corresponding author. Email: yaojunc@nju.edu.cn}~, Yanbo ZHANG$^2$\\
{\small $^1$School of Mathematics, Nanjing University, Nanjing 210093, China}\\
{\small $^2$School of Mathematical Sciences, Hebei Normal University, Shijiazhuang 050024, China}}
\date{}
\maketitle
\begin{abstract} Let $\mathcal{F}$ be a $k$-uniform hypergraph. The famous Erd\H{o}s--Ko--Rado (1961) theorem  determines the maximum size and extremal structure for $\mathcal{F}$ being $t$-intersecting, that is, 
$|F_1 \cap F_2| \ge t$ for any two edges $F_1, F_2$ of $\mathcal F$. The codegree squared sum $\mathrm{co}_2(\mathcal{F})$ is the square of the $\ell_2$-norm of the codegree vector of all $(k-1)$-sets in $\mathcal{F}$, which was initially introduced for Tur\'an problems of hypergraphs. Recently, Brooks and Linz (2026), as well as Wu and Zhang (2026) investigated 
the maximum value of $\mathrm{co}_2(\mathcal{F})$ and corresponding extremal structures for $\mathcal{F}$ being $t$-intersecting. Moreover, Brooks and Linz asked if the classical results on intersecting families can be extended to $\mathrm{co}_2(\mathcal{F})$. In this paper, by developing the spectral techniques for incidence matrices, we study the extremal problems of $\mathrm{co}_2(\mathcal{F})$ for $\mathcal{F}$ being intersecting families in finite vector spaces, affine spaces, and attenuated spaces, and establish the Erd\H{o}s--Ko--Rado theorems in $\ell_2$-norm for the three finite spaces.

\end{abstract}

\section{Introduction}


A major goal of extremal combinatorics is to understand the largest size of discrete structures satisfying prescribed constraints. For $k$-uniform hypergraphs, this leads to Tur\'{a}n-type problems, which ask for the maximum number of edges in a hypergraph avoiding a fixed family of forbidden configurations. Despite substantial progress, these problems remain widely open, particularly for $k\geq 3$. This has motivated the study of finer structural parameters beyond the edge count, such as codegrees and the $\ell_2$-norm of the codegree vector.

Let $[n] = \{1,2,\dots,n\}$, and let $\binom{[n]}{k}$ denote the family of all $k$-element subsets of $[n]$. For a $k$-uniform hypergraph $\mathcal{F} \subseteq \binom{[n]}{k}$, the \emph{codegree} $d(E)$ of a set $E \subseteq [n]$ with $|E| = k-1$ is defined as the number of edges of $\mathcal{F}$ that contain $E$, namely
\[
d(E) = \left| \bigl\{ F \in \mathcal{F} : E \subseteq F \bigr\} \right|.
\]
The vector of all codegrees over all $(k-1)$-subsets forms a vector $\mathbf{x} = (d(E))_{E \in \binom{[n]}{k-1}}$. The \emph{codegree squared sum} of $\mathcal{F}$, denoted by $\co(\mathcal{F})$, is defined as the square of the $\ell_2$-norm of this codegree vector:
\[
\co(\mathcal{F}) = \sum_{E \in \binom{[n]}{k-1}} d(E)^2.
\]
Since $\sum_{E \in \binom{[n]}{k-1}} d(E) = k |\mathcal{F}|$, the $\ell_1$-norm of the codegree vector is proportional to the classical edge count, while the $\ell_2$-norm additionally captures the concentration of codegrees, thereby revealing more refined structural properties of the hypergraph.

A fundamental class of $k$-uniform hypergraphs is given by \emph{$t$-intersecting} families, in which any two edges meet in at least $t$ vertices. The classical Erd\H{o}s--Ko--Rado \cite{EKR1961} theorem  determines the maximum size of a $t$-intersecting family for sufficiently large $n$, and corresponding extremal structure.
 
\begin{theorem}(Erd\H{o}s-Ko-Rado \cite{EKR1961})
Let $t,k,n$ be positive integers such that $t \leq k \leq n$. If $\mathcal{F}\subseteq \binom{[n]}{k}$ is $t$-intersecting, then there is a function $n_{0}(k,t)$ such that 
$$|\mathcal{F}|\leq\binom{n-t}{k-t}$$ for $n\geq n_{0}(k,t)$. Equality holds only if $\mathcal{F}=\{F\in\binom{[n]}{k}: \ X\subseteq F\}$ for some $t$-subset $X$.
\end{theorem}

This result provides a natural direction for studying the extremal problem in $\ell_2$-norm  of $t$-intersecting families. Recently, Balogh, Clemen, and Lidick\'{y} \cite{Balogh2022} introduced an $\ell_2$-extremal framework for hypergraphs: given a forbidden family $\mathcal{H}$ of $k$-uniform hypergraphs, one maximizes $\co(\mathcal{F})$ over all $\mathcal{H}$-free $k$-uniform hypergraphs $\mathcal{F}$. This framework is closely related to Bey's general upper bound on the sum of squared codegrees \cite{Bey2003}, and has motivated further results, including the first Erd\H{o}s--Ko--Rado type theorem in $\ell_2$-norm for $t=1$, due to Brooks and Linz \cite{Brooks2023}. In the same paper, they further established analogues of the Erd\H{o}s matching conjecture and the  Erd\H{o}s--Ko--Rado theorem in $\ell_2$-norm for $t$-intersecting families, valid for sufficiently large $n$.

In the $\ell_2$-norm setting, Brooks and Linz \cite{Brooks2023} raised a natural question: can the  classical results on intersecting families, including both the EKR and FHM theorems, be extended to the codegree squared sum? For the EKR case, this was fully resolved by Wu and Zhang \cite{WuZhang2026}, who obtained the following.

\begin{theorem}(Wu-Zhang \cite{WuZhang2026})
Let $t,k,n$ be positive integers such that $t \leq k \leq n$. If $\mathcal{F}\subseteq\binom{[n]}{k}$ is $t$-intersecting, then 
for $n\geq(t+1)(k-t+1)$, we have
\[
\co(\mathcal{F})\leq\binom{n-t}{k-t}(t+(n-k+1)(k-t))    
\]
equality holds if and only if $\mathcal{F}=\left\{F\in\binom{[n]}{k}: \ X\subseteq F\right\}$ for some $t$-subset $X$ of $[n]$. Moreover, when $n=2k$ and $t=1$, then the family $\binom{Z}{k}$ is also extremal, where $Z$ is a $(n-1)$-subset of $[n]$.
\end{theorem}

The Erd\H{o}s-Ko-Rado (EKR) theorem is a foundational result in extremal set theory. Its classical form has been extensively studied, and its generalizations have substantially shaped several areas of combinatorics. The introduction of the $\ell_2$-norm version of the EKR theorem naturally raises the question of whether analogous $\ell_2$-norm EKR-type extremal results can be established for other combinatorial structures. This paper addresses this question by studying the $\ell_2$-norm EKR extremal properties of three  finite spaces: vector spaces, affine spaces, and attenuated spaces. 

\subsection{Intersecting families and EKR theorems in the three  spaces}
In this section, we give the definitions of the intersecting families $\mathcal{F}$ and $\mathrm{co}_2(\mathcal{F})$ in the three spaces, respectively. Moreover, we also present the  EKR theorems in these three spaces.

\vskip 2mm
Recall that for any positive integers $a\geq b$, the Gaussian binomial coefficient is defined by
\[
 \qbinom{a}{b}=\prod\limits_{i=0}^{b-1}\frac{q^{a-i}-1}{q^{b-i}-1}.
\]
In addition, we set $\qbinom{a}{b}=0$ if $a<b$ and $\qnum{a}=\qbinom{a}{1}$.
The Gaussian binomial coefficients satisfy
\begin{equation}\label{eq:qbinom-ratio}
 \qbinom{N}{r-1}=\qbinom{N}{r}\frac{\qnum r}{\qnum{N-r+1}},
\end{equation}
and the $q$-Pascal identity
\begin{equation}\label{eq:q-pascal}
 \qbinom{N}{r}=\qbinom{N-1}{r-1}+q^r\qbinom{N-1}{r},
\end{equation}
and
\begin{equation}\label{eq:qinteger-sum}
[a+b]_q=[a]_q+q^a[b]_q.
\end{equation}

\vskip 2mm
\noindent {\bf 1. Vector spaces over a finite field}
\vskip 2mm

Let $q$ be a prime power and $V$ be an $n$-dimensional vector space over $\Fq$. For $0\le k\le n$, we use $\Gr{V}{k}$ to denote the family of all $k$-dimensional subspaces of $V$. Then it is well known that $\qbinom{n}{r}=|\Gr{V}{r}|$. In what follows we will abbreviate ``$k$-dimensional subspace" to ``$k$-subspace".  
A family $\mathcal F\subseteq \Gr{V}{k}$ is \emph{intersecting} if
$\dim(F\cap G)\ge 1$ for all $F,G\in\mathcal F$. For $E\in\Gr{V}{k-1}$, define its codegree and codegree squared sum as follows:
\[
d_{\mathcal F}(E)=|\{F\in\mathcal F:E\subseteq F\}|,
\qquad
\co(\mathcal F)=\sum_{E\in\Gr{V}{k-1}}d_{\mathcal F}(E)^2.
\]

Using combinatorial methods, Hsieh \cite{Hsieh1975} proved the EKR theorem for finite vector spaces as follows.

\begin{theorem}(Hsieh \cite{Hsieh1975})\label{thm:Hsieh}
If $n>2k$ and $\mathcal F\subseteq \Gr{V}{k}$ is intersecting, then
\[
 |\mathcal F|\le \qbinom{n-1}{k-1}.
\]
Equality holds if and only if  $\mathcal F=\{F\in \Gr{V}{k}:\ X\subseteq F\}$ for some 1-subspace $X\subseteq V$.
\end{theorem}

\noindent {\bf 2. Finite affine spaces}
\vskip 2mm

Let $V$ be an $n$-dimensional vector space over $\Fq$ and let $\AG(n,q)$ be the affine space with translation space $V$.  A $k$-flat is a coset
\[
P=a+U,
\quad a\in V,
\quad U\subseteq V,
\quad \dim (U)=k.
\]
Its direction is $\dir(P)=U$. Let $P=a+U_1$ and $Q=b+U_2$ be two flats in $\AG(n,q)$, with direction subspaces $U_1, U_2\subseteq V$, respectively. We say that $P$ and $Q$ are \emph{parallel}, denoted by $P\parallel Q$ if $U_1=U_2$. Moreover, $P\subseteq Q$ if and only if $U_1\subseteq U_2$ and $a-b\in U_2$.
 
Let $\mathcal M_k$ be the set of all $k$-flats. Then $|\mathcal M_k|=q^{n-k}\qbinom{n}{k}$,
because there are $\qbinom{n}{k}$ choices for $U$ and, for fixed $U$, there are $q^{n-k}$ cosets of $U$ in $V$. A family $\mathcal{F}\subseteq {\mathcal M_k}$ is \emph{intersecting} if $F_1 \cap F_2 \neq \emptyset$ for any $F_1,F_2\in {\mathcal F}$. Note that the intersection $F_1 \cap F_2$ is also a flat, and its dimension satisfies $\dim(F_1 \cap F_2) \ge 0$. If $F_1 \cap F_2 = \emptyset$, then we define $\dim(F_1 \cap F_2)= -1$. For $E\in\mathcal M_{k-1}$, define
\[
d_{\mathcal F}(E)=|\{P\in\mathcal F:E\subseteq P\}|,
\qquad
\operatorname{co}^{\mathrm{AG}}_2(\mathcal F)=\sum_{E\in\mathcal M_{k-1}}d_{\mathcal F}(E)^2.
\]

Guo and Xu \cite{GuoXu2017} generalized Hsieh’s result and obtained
the EKR theorem for finite affine spaces.

\begin{theorem}(Guo-Xu \cite{GuoXu2017})\label{thm:guo-xu}
Assume $n\ge 2k+1$.  If $\mathcal F\subseteq\mathcal M_k$ is an intersecting family of $k$-flats in $\AG(n,q)$, then
\[
|\mathcal F|\le \qbinom{n}{k}.
\]
Equality holds if and only if $\mathcal{F}$ consists of all $k$-flats containing a fixed $0$-flat in $\AG(n,q)$.
\end{theorem}


\noindent {\bf 3. Attenuated spaces}
\vskip 2mm

Let $V$ be an $(n+\ell)$-dimensional vector space over $\Fq$, and fix an $\ell$-subspace $W\subseteq V$. A \emph{maximal attenuated subspace} is an $n$-subspace $P$ such that $P\cap W=\{\bm{0}\}$. Denote the set of all such subspaces by $\mathcal M_n=\{P\subseteq V:\dim (P)=n,\ P\cap W=\{\bm{0}\}\}$. A family $\mathcal F\subseteq\mathcal M_n$ is \emph{intersecting} if
$\dim(P\cap Q)\ge1$ for any $P,Q\in\mathcal F$.

Let $N$ be a complement of $W$, so $V=N\oplus W$ and $\dim (N)=n$.  For each $P\in\mathcal M_n$, there exists a unique linear map $A:N\to W$ such that $P=P_A=\{x+A(x):x\in N\}$, whose justification is deferred to Lemma \ref{lem:attenuated-bey}. Thus $\mathcal M_n$ is naturally identified with the additive group $G=\Hom(N,W)$ and $|G|=q^{\ell n}$.

Let $\mathcal M_{n-1}=\{E\subseteq V:\dim (E)=n-1,\ E\cap W=\{\bm{0}\}\}$. For $E\in\mathcal M_{n-1}$, define
\[
d_{\mathcal F}(E)=|\{P\in\mathcal F:E\subseteq P\}|,
\qquad
\operatorname{co}^{\mathrm{AT}}_2(\mathcal F)=\sum_{E\in\mathcal M_{n-1}}d_{\mathcal F}(E)^2.
\]
Let $H\subseteq N$ be a given $(n-1)$-subspace (called a hyperplane hereafter). Then every $E\in\mathcal M_{n-1}$ has a unique form $E_{H,a}=\{x+a(x):x\in H\}$, where $a:H\to W$ is linear. 
The subspace $P_A$ contains $E_{H,a}$ if and only if $A|_H=a$. Since $N/H$ is one-dimensional, there are exactly $|\Hom(N/H,W)|=q^\ell$ extensions $A:N\to W$ of $a$.

The following EKR theorem in attenuated space is due to Huang \cite{Huang1987}, which is presented in different form in the original paper.


\begin{theorem}(Huang \cite{Huang1987})\label{thm:Huang}
Assume $n\ge2$ and $\ell\ge n$.  If $\mathcal F\subseteq\mathcal M_n$ is intersecting, then
\[
|\mathcal F|\le q^{\ell(n-1)}.
\]
If $\ell>n$, equality holds exactly for the families consisting of all members of $\mathcal M_n$ that contain a fixed $1$-subspace $L\subseteq V$ with $L\cap W=\{\bm{0}\}$. If $\ell=n$, besides the above family, equality also holds exactly for the dual families
\[
\mathcal D_U=\{P\in\mathcal M_n:P\subseteq U\},
\]
where $U\subseteq V$, $\dim (U)=2n-1$ and $\dim(U\cap W)=n-1$.
\end{theorem}

\subsection{Main results}
In this paper, we try to establish the EKR theorems in $\ell_2$-norm for the intersecting families listed in Theorems \ref{thm:Hsieh}-\ref{thm:Huang}. The main results are the following three theorems.

\begin{theorem}\label{thm:vector-l2-ekr}
Assume $n>2k$.  If $\mathcal F\subseteq\Gr{V}{k}$ is intersecting, then
\begin{equation}\label{eq:vector-l2-main}
\co(\mathcal F)
\le
\qbinom{n-1}{k-1}
\bigl([k-1]_q[n-k+1]_q+q^{k-1}\bigr).
\end{equation}
Equality holds if and only if $\mathcal F=\{F\in \Gr{V}{k}:\ X\subseteq F\}$ for some 1-subspace $X\subseteq V$.
\end{theorem}

\begin{theorem}\label{thm:affine-l2-ekr}
Assume $k\ge 1$ and $n\ge2k+1$.  If $\mathcal F\subseteq\mathcal M_k$ is an intersecting family of $k$-flats in $\AG(n,q)$, then
\begin{equation}\label{eq:affine-main-first}
\operatorname{co}^{\mathrm{AG}}_2(\mathcal F)
\le q[k]_q\qbinom{n}{k}\bigl([n-k]_q+1\bigr).
\end{equation}
Equivalently,
\begin{equation}\label{eq:affine-main-second}
\operatorname{co}^{\mathrm{AG}}_2(\mathcal F)
\le \qbinom{n}{k-1}\bigl([n-k+1]_q^2+q^{n-k+1}-1\bigr).
\end{equation}
Equality holds if and only if $\mathcal F$ consists of all $k$-flats containing a fixed $0$-flat (i.e., a point) in $\AG(n,q)$.
\end{theorem}

\begin{theorem}\label{thm:attenuated-l2-ekr}
Assume $n\ge2$ and $\ell\ge n$.  If $\mathcal F\subseteq\mathcal M_n$ is intersecting, then
\begin{equation}\label{eq:attenuated-main}
\operatorname{co}^{\mathrm{AT}}_2(\mathcal F)
\le q^{\ell(n-1)}\bigl(q^\ell[n-1]_q+q^{n-1}\bigr).
\end{equation}
If $\ell>n$, equality holds exactly for the families consisting of all members of $\mathcal M_n$ that contain a fixed $1$-subspace $L\subseteq V$ with $L\cap W=\{\bm{0}\}$.  If $\ell=n$, besides the above family, equality also holds exactly for the dual families in Theorem \ref{thm:Huang}.
\end{theorem}

The remainder of this paper is organized as follows. In Section 2, by defining a new incidence matrix, we convert the upper bound of the $\ell_2$-norm under consideration into a function of the eigenvalues of the incidence matrix and the size of the family, which provides a unified framework for proving Theorems \ref{thm:vector-l2-ekr}-\ref{thm:attenuated-l2-ekr}. 
Section 3 introduces some additional tools concerning groups and the spectrum of some incidence operator. Sections 4-6 are devoted to the three finite spaces, respectively.

\section{A key lemma}

Let $\mathscr{X}$ be the family of all $k$-subspaces, $k$-flats, $n$-subspaces $L$ with $L\cap W=\{\bm{0}\}$, and $\mathscr{Y}$ be the family of all $(k-1)$-subspaces, $(k-1)$-flats, $(n-1)$-subspaces $L'$ with $L'\cap W=\{\bm{0}\}$ in the three spaces, respectively. Let $B=(B_{X,Y})_{|\mathscr{X}|\times |\mathscr{Y}|}$ be the $0$-$1$ incidence matrix with rows indexed by $\mathscr{X}$ and columns indexed by $\mathscr{Y}$ such that for any $X\in \mathscr{X}$ and $Y\in \mathscr{Y}$,  $B_{X,Y}=1$ if and only if $Y\subseteq X$. Set $T=BB^T$.

\begin{lemma}\label{lem:general-spectral}

Suppose the all-one vector $\mathbf 1\in\mathbb R^{|\mathscr{X}|}$ is an eigenvector of $T$ with eigenvalue $\lambda_0$, and every vector $\bm{z}$ orthogonal to $\mathbf 1$ has Rayleigh quotient at most $\lambda_1$, i.e.,
\[
\frac{\langle \bm{z},T\bm{z}\rangle}{\langle \bm{z},\bm{z}\rangle}\le \lambda_1
\qquad(\bm{z}\perp \mathbf 1,\\ \bm{z}\ne 0).
\]
If $\mathcal A\subseteq \mathscr{X}$, $m=|\mathcal A|$, and $v=|\mathscr{X}|$, then
\[
\co(\mathcal A)\le \lambda_1m+(\lambda_0-\lambda_1)\frac{m^2}{v}.
\]
\end{lemma}

\begin{proof}
Let $\bm{x}\in\{0,1\}^{|\mathscr{X}|}$ be the characteristic vector of $\mathcal A$.  The $Y$-th coordinate of $B^T\bm{x}$ is exactly
\[
(B^T\bm{x})_Y=\sum_{X\in \mathscr{X}}B_{X,Y}\bm{x}_X=|\{X\in\mathcal A:Y\subseteq X\}|=d_{\mathcal A}(Y),
\]
so
\[
\co(\mathcal A)=\|B^T\bm{x}\|^2=\bm{x}^TBB^T\bm{x}=\bm{x}^TT\bm{x}.
\]
Decompose $\bm{x}$ into a vector in $span\{\mathbf{1}\}$ and a vector in its orthogonal complement:
\[
\bm{x}=\frac{m}{v}\mathbf 1+\bm{y},
\qquad \bm{y}\perp \mathbf 1.
\]
Since $T\mathbf 1=\lambda_0\mathbf 1$ and $T$ is real symmetric, we have
\[
\left\langle \frac{m}{v}\mathbf 1,T\bm{y}\right\rangle
=\frac{m}{v}\langle T\mathbf 1,\bm{y}\rangle
=\frac{m}{v}\lambda_0\langle \mathbf 1,\bm{y}\rangle=0.
\]
Moreover,
\[
\left\|\frac{m}{v}\mathbf 1\right\|^2=\frac{m^2}{v},
\qquad
\|\bm{y}\|^2=\|\bm{x}\|^2-\frac{m^2}{v}=m-\frac{m^2}{v}.
\]
Therefore,
\[
\bm{x}^TT\bm{x}=\lambda_0\frac{m^2}{v}+\bm{y}^TT\bm{y}
\le \lambda_0\frac{m^2}{v}+\lambda_1\left(m-\frac{m^2}{v}\right).
\]
The proof is complete.
\end{proof}

\section{Auxiliary Tools}

This section reviews some known results that will be used to calculate the spectrum  of the matrix $T=BB^T$ defined on the three spaces, which together with Lemma \ref{lem:general-spectral} can help us deduce the expected upper bounds.

\vskip 2mm
\noindent{\bf 1. Finite abelian group: Actions, Characters, and Operators}
\vskip 2mm

All finite abelian groups are written additively. A finite-dimensional vector space $V$ over $\Fq$ is regarded as a finite abelian group $(V,+)$ under addition. 

Let $S$ be a finite set and let $\mathbb{C}^S$ denote the vector space of all complex-valued functions on $S$. We equip $\mathbb{C}^S$ with the standard Hermitian inner product 
\[
    \langle f,g\rangle_S=\sum\limits_{x\in S}f(x)\overline{g(x)}, \quad f,g\in \mathbb{C}^S.
\]
\begin{definition}
A character of a finite additive group $G$ is a group homomorphism $\chi : G \to \mathbb{C}^\times$, thus  
\[\chi(x + y) = \chi(x)\chi(y) \quad (x, y \in G),\]
where $\mathbb{C}^\times=\mathbb{C}\setminus\{0\}$. The character identically equal to 1 is called trivial; all other characters are called nontrivial. 
\end{definition}
Since $G$ is finite, every character takes values on the unit circle, and therefore  
\[\chi(-x) = \chi(x)^{-1} = \overline{\chi(x)}.\]
An additive character of $\Fq$ is a character of the additive group $(\Fq,+)$. Fix throughout a nontrivial additive character $\psi : (\Fq,+) \to \mathbb{C}^\times$. 
\begin{lemma}\label{lem: character sum}(\cite{Lidl})
Let $G$ be a finite abelian group, and let $\chi$ be a character of $G$. Then
\[\sum_{x \in G} \chi(x) = 
\begin{cases} 
|G|, & \chi \text{ is trivial,} \\
0, & \chi \text{ is nontrivial.}
\end{cases}
\]
More generally, for every subgroup $H \leq G$,
\[\sum_{x \in H} \chi(x) = 
\begin{cases} 
|H|, & \chi|_H \text{ is trivial,} \\
0, & \chi|_H \text{ is nontrivial.}
\end{cases}
\]
Consequently, for two characters $\chi, \xi$ of $G$,
\[\sum_{x \in G} \chi(x) \overline{\xi(x)} = 
\begin{cases} 
|G|, & \chi = \xi, \\
0, & \chi \neq \xi.
\end{cases}
\]
\end{lemma}
\begin{definition}
Let $V$ and $V'$ be finite-dimensional vector spaces over $\Fq$. A bilinear pairing $\beta : V' \times V \to \Fq$ is called a \textbf{perfect pairing} if the induced linear maps
\[V \to (V')^*, \ t \mapsto \beta(\cdot,t) \quad \text{and} \quad V' \to V^*, \ s \mapsto \beta(s, \cdot),\]
are isomorphisms, where $(V')^*=\Hom(V',\Fq)$ and $V^*=\Hom(V,\Fq)$. Equivalently, $\beta$ is nondegenerate in both variables: 
\[\beta(s, t) = 0 \text{ for all } s \in V' \implies t = 0,\]
and
\[\beta(s, t) = 0 \text{ for all } t \in V \implies s = 0.\]
\end{definition}
\begin{lemma}\label{lem:perfect pairing}(\cite{Ceccherini2008})
Let $\beta : V' \times V \to \Fq$ be a perfect bilinear pairing. For $s\in V'$, define
\[\chi_s(x) := \psi(\beta(s, x)) \quad (x\in V).\]
Then the functions $\{\chi_s : s \in V'\}$ are precisely the characters of the additive group $V$. Moreover, they form an orthogonal basis of $\mathbb{C}^V$.
\end{lemma}

Let $G$ be a finite abelian group acting on a finite set $S$ with the action written on the right as $x\cdot t$ $(x \in S, t \in G)$. For $t\in G$, define the induced operator 
\[\rho_t : \mathbb{C}^S \to \mathbb{C}^S, \quad (\rho_t f)(x) := f(x \cdot t).\]
Then, $\rho_s \rho_t = \rho_{s+t}$, $\rho_t^{-1}=\rho_{-t}$. Moreover, each $\rho_t$ is unitary with respect to the standard Hermitian inner product on $\mathbb{C}^S$. Indeed, since $x \mapsto x \cdot t$ is a permutation of $S$, one has $\langle \rho_tf,\rho_tg\rangle_S=\langle f,g\rangle_S$. Let $\widehat G=\operatorname{Hom}(G,\mathbb C^\times)$ denote the character group of $G$.

\begin{lemma}\label{lem:character decomposition}(\cite{Ceccherini2008})
For each $\chi\in\widehat G$, define the corresponding character component by
\[\mathcal{C}_s := \{ f \in \mathbb{C}^S : \rho_t f = \chi_s(t) f \text{ for every } t \in G \}.\]
Then\[\mathbb{C}^S = \bigoplus_{s \in \widehat G} \mathcal{C}_s.\]
If a linear operator $T : \mathbb{C}^S \to \mathbb{C}^S$ satisfies $T \rho_t = \rho_t T$ for $t \in G$, then every $\mathcal{C}_s$ is $T$-invariant, i.e., $T(\mathcal{C}_s)\subseteq\mathcal{C}_s$.
\end{lemma}

\begin{lemma}\label{lem: convolution}(\cite{Ceccherini2008})
Given a finite abelian group $G$, for $h, f \in \mathbb{C}^G$, define their unnormalized convolution by
\[
(h * f)(x) := \sum_{y \in G} h(x - y)f(y).
\]
Let $T : \mathbb{C}^G \to \mathbb{C}^G$ have matrix entries $T_{x,y} = h(x - y)$. Then $Tf = h * f$. Moreover, every character $\chi$ of $G$ is an eigenvector of $T$, with eigenvalue
\[
\lambda_\chi = \sum_{z \in G} h(z)\chi(-z) = \sum_{y \in G} T_{0,y}\chi(y).
\]
Thus the characters of $G$ form an orthogonal eigenbasis for $T$.
\end{lemma}

\noindent{\bf 2. The spectrum of incidence operator}
\vskip 2mm

For each $r$, let $\mathcal W_r$ be the vector space consisting of all real value functions on $\Gr{V}{r}$, equipped with the inner product
\[
\langle f,g\rangle=\sum_{R\in\Gr{V}{r}}f(R)g(R).
\]
Define the up and down operators by
\[
U_r:\mathcal W_r\to\mathcal W_{r+1},
\qquad
(U_rf)(X)=\sum_{\substack{R\subseteq X\\ \dim (R)=r}}f(R),
\]
and

\[
D_{r+1}:\mathcal W_{r+1}\to\mathcal W_r,
\qquad
(D_{r+1}g)(R)=\sum_{\substack{X\supseteq R\\ \dim (X)=r+1}}g(X).
\]
The background of up and down operators can be found in \cite{Stanley1991}. Define 
$$M_r=U_{r-1}D_r$$
as the incidence operator.  
The spectrum of $M_r$ is determined by Fulman \cite{Fulman2009}. 
\begin{lemma}(Fulman \cite{Fulman2009})\label{prop:grassmann-spectrum}
Assume $1\le k\le n/2$.  On $\mathcal W_k$, let $M_k=U_{k-1}D_k$. Then the eigenvalues of $M_k$ are
\[
    \lambda_i=q^i[k-i]_q[n-k-i+1]_q,
\qquad i=0,1,\ldots,k.
\]
The eigenspace for $\lambda_0$ is the one-subspace composed of constant functions, and the largest eigenvalue on the orthogonal complement of this 1-subspace is
\[
    \lambda_1=q[k-1]_q[n-k]_q.
\]
Moreover, $\lambda_0-\lambda_1=[n]_q$.
\end{lemma}

\section{Vector spaces over a finite field}
 
The goal of this section is to prove Theorem \ref{thm:vector-l2-ekr}. For this purpose, we establish a Bey-type inequality in vector space, and finally give the proof in details.  

\begin{lemma}\label{lem:vector-bey}
Let $\mathcal A\subseteq\Gr{V}{k}$ and $m=|\mathcal A|$.  If $1\le k\le n/2$, then
\[
\co(\mathcal A)\le q[k-1]_q[n-k]_qm+[n]_q\frac{m^2}{\qbinom{n}{k}}.
\]
\end{lemma}

\begin{proof}
Let $\mathscr{X}=\Gr{V}{k}$ and $\mathscr{Y}=\Gr{V}{k-1}$, and let $B$ be as defined before. Then $BB^T$ is precisely the incidence operator $M_k=U_{k-1}D_k$ on functions on $\Gr{V}{k}$: starting from a $k$-subspace, it sums over all $(k-1)$-subspaces the $k$-subspace contains, and then over all $k$-subspaces containing those.

By Lemma \ref{prop:grassmann-spectrum}, the all-one vector $\mathbf{1}$ corresponds to the eigenvalue $\lambda_0=[k]_q[n-k+1]_q$, while among the remaining eigenvalues, the largest one is $\lambda_1=q[k-1]_q[n-k]_q$. Moreover, $\lambda_0-\lambda_1=[n]_q$.  Applying Lemma \ref{lem:general-spectral} with $v=|\mathscr{X}|=\qbinom{n}{k}$, the inequality holds.
\end{proof}

\vskip 3mm
\begin{proof}[\bfseries{Proof of Theorem \ref{thm:vector-l2-ekr}}]
Let $m=|\mathcal F|$ and $M=\qbinom{n-1}{k-1}$. By Theorem \ref{thm:Hsieh}, we have $m\le M$.  Lemma \ref{lem:vector-bey} gives

\[
\co(\mathcal F)\le \mu_1(m),
\qquad
\mu_1(x)=q[k-1]_q[n-k]_qx+[n]_q\frac{x^2}{\qbinom{n}{k}}.
\]
The function $\mu_1$ is strictly increasing on $[0,\infty)$: if $0\le x_1<x_2$, then

\[
\mu_1(x_2)-\mu_1(x_1)=q[k-1]_q[n-k]_q(x_2-x_1)+[n]_q\frac{x_2^2-x_1^2}{\qbinom{n}{k}}>0.
\]
Therefore
\[
\co(\mathcal F)\le \mu_1(M)
=q[k-1]_q[n-k]_qM+[n]_q\frac{M^2}{\qbinom{n}{k}}.
\]
Using $\frac{M}{\qbinom{n}{k}}=\frac{\qbinom{n-1}{k-1}}{\qbinom{n}{k}}=\frac{[k]_q}{[n]_q}$, we get $[n]_q\frac{M^2}{\qbinom{n}{k}}=M[k]_q$. Thus
\[
\co(\mathcal F)
\le M\bigl(q[k-1]_q[n-k]_q+[k]_q\bigr).
\]
Since $[k]_q=[k-1]_q+q^{k-1}$ and $[n-k+1]_q=1+q[n-k]_q$, we have
\[
q[k-1]_q[n-k]_q+[k]_q
=[k-1]_q[n-k+1]_q+q^{k-1}.
\]
This proves the upper bound.

It remains to verify when the equality holds.  Fix a $1$-subspace $X$ of $V$ and consider
\[
\mathcal S_X=\{F\in\Gr{V}{k}:X\subseteq F\}.
\]
Let $E\in\Gr{V}{k-1}$.  If $X\subseteq E$, then every $k$-subspace containing $E$ also contains $X$, so $d_{\mathcal S_X}(E)=[n-k+1]_q$.
The number of such $E$ is $\qbinom{n-1}{k-2}$. If $X\not\subseteq E$, then $E\cap X=\{\bm{0}\}$, so $\dim(E+X)=k$.  Any $k$-subspace containing both $E$ and $X$ must contain $E+T$ and hence must equal $E+X$.  Thus, $d_{\mathcal S_X}(E)=1$. By the $q$-Pascal identity \eqref{eq:q-pascal}, the number of $(k-1)$-subspaces not containing $X$ is
\[
\qbinom{n}{k-1}-\qbinom{n-1}{k-2}=q^{k-1}\qbinom{n-1}{k-1}.
\]
Therefore
\[
\co(\mathcal S_X)
=\qbinom{n-1}{k-2}[n-k+1]_q^2
+q^{k-1}\qbinom{n-1}{k-1}.
\]
Using \eqref{eq:qbinom-ratio} with $N=n-1$ and $r=k-1$ gives

\[
\qbinom{n-1}{k-2}
=\qbinom{n-1}{k-1}\frac{[k-1]_q}{[n-k+1]_q}.
\]
Substitution yields

\[
\co(\mathcal S_X)
=\qbinom{n-1}{k-1}
\bigl([k-1]_q[n-k+1]_q+q^{k-1}\bigr),
\]
which is exactly the right-hand side of \eqref{eq:vector-l2-main}.

If equality holds in \eqref{eq:vector-l2-main} for some intersecting family $\mathcal F$, then the strict monotonicity of $\mu_1$ forces $m=M$.  Hsieh's equality statement in Theorem \ref{thm:Hsieh} then forces $\mathcal F=\mathcal S_X$ for some $1$-subspace $X$. Conversely, the computation above shows that every $\mathcal S_X$ attains the equality.
\end{proof}

\section{Finite affine spaces}

Before starting to prove Theorem \ref{thm:affine-l2-ekr}, we  need to establish Bey-type inequality for affine spaces. As a preparation, we give two lemmas.

The following lemma  counts the number of some $(k-1)$-flats and $k$-flats.

\begin{lemma}\label{lem:affine-counts}
Every $k$-flat contains exactly $q[k]_q$ $(k-1)$-flats, and every $(k-1)$-flat is contained in exactly $[n-k+1]_q$ $k$-flats.
\end{lemma}

\begin{proof}
Let $P=a+U$ be a $k$-flat with $\dim (U)=k$. To count the number of $(k-1)$-flats contained in $P$, we first choose its direction subspace, each being a $(k-1)$-subspace $H\subseteq U$, and there are $[k]_q$ such choices. For each fixed $H$, the $(k-1)$-flats contained in $P$ with direction $H$ are precisely the affine subspaces $a+u+H, u+H\in U/H$. Thus, their number is $|U/H|=q$.  Therefore, $P$ contains exactly $q[k]_q$ $(k-1)$-flats.

Let $E=b+H$ be a $(k-1)$-flat and $\dim (H)=k-1$.  Every $k$-flat containing $E$ can be denoted as $b+U$, where $U$ is a $k$-subspace satisfying $H\subseteq U\subseteq V$. Such subspaces $U$ are in bijection with one-subspaces $U/H$ of $V/H$. Since $\dim(V/H)=n-k+1$, the number of such subspaces is $[n-k+1]_q$, i.e., $E$ is contained in exactly $[n-k+1]_q$ $k$-flats.
\end{proof}

The canonical evaluation pairing
\[
V^* \times V \to \Fq, \quad (\eta, t) \longmapsto \eta(t),
\]
is perfect. Indeed, a nonzero linear functional does not vanish on all of $V$, while for every nonzero $t \in V$ there exists $\eta \in V^*$ such that $\eta(t) \neq 0$. Thus the pairing is nondegenerate in both variables, and $\dim(V)=\dim(V^*)$ implies that it is perfect. By Lemma \ref{lem:perfect pairing}, the characters of the additive group $V$ are therefore precisely $\chi_\eta(t) = \psi(\eta(t))$ for $\eta \in V^*$. We now apply Lemma \ref{lem:character decomposition} to the translation action of $V$ on $\mathcal{M}_k$.

\begin{lemma}\label{lem:affine-character-components}
Let the additive group $V$ act on $\mathcal M_k$ by $(a+U)\cdot t=(a+t)+U$. For $\eta\in V^*=\Hom(V,\Fq)$, let $\chi_\eta(t)=\psi(\eta(t))$ and let $\mathcal C_\eta$ be the corresponding character component and $K_\eta=\ker\eta$. Then $f\in\mathcal C_\eta$ if and only if there exists a function $g:\Gr{K_\eta}{k}\to\mathbb C$ such that

\begin{equation}\label{eq:f(a+U)}
    f(a+U)=
\begin{cases}
\psi(\eta(a))g(U),&U\subseteq K_\eta,\\
0,&U\nsubseteq K_\eta.
\end{cases}
\end{equation}
Moreover, $\Phi_\eta: \mathbb C^{\Gr{K_\eta}{k}}\to\mathcal C_\eta$
defined by the preceding formula \eqref{eq:f(a+U)} is a linear isomorphism with inverse $(\Phi^{-1}_\eta f)(U)=f(0+U)$. Moreover,
\[
\mathbb C^{\mathcal M_k}
=
\bigoplus_{\eta\in V^*}^{\perp}\mathcal C_\eta.
\]    
\end{lemma}
\begin{proof}
Suppose first that $f \in \mathcal{C}_\eta$, and let $P = a + U$. For every $u\in U$, $P \cdot u = P$,
\[f(P) = f(P \cdot u) = \chi_\eta(u) f(P)=\psi(\eta(u))f(P).\] If $U\nsubseteq \ker \eta$, then $\eta|_U: U\to \Fq$ is nonzero and therefore surjective onto $\Fq$. Since $\psi$ is nontrivial, there exists $u\in U$ such that $\chi_\eta(u) \neq 1$, then forces $f(P)=0$.

Now suppose that $U\subseteq\ker\eta$. Define $g(U):=f(0+U)$. Since $a+U=(0+U)\cdot a$ we have $f(a + U) = \chi_\eta(a)f(0 + U) = \psi(\eta(a))g(U)$. This proves the necessity of \eqref{eq:f(a+U)}. Conversely, let $g: \Gr{K_\eta}{k}\to \mathbb{C}$ be arbitrary, and define $f=\Phi_\eta g$. This definition is independent of the representative $a$. Indeed, if $a + U = a' + U$ and $U \leq K_{\eta}$, then $a' - a \in U \subseteq K_{\eta}$, so $\eta(a') = \eta(a)$; if $U \subseteq K_{\eta}$, then for every $t \in V$, 
\[ f((a + U) \cdot t) = f((a + t) + U) = \psi(\eta(a + t))g(U)= \chi_{\eta}(t)f(a + U).\]
If $U \nsubseteq K_{\eta}$, both sides are zero. Thus, $\rho_t f = \chi_{\eta}(t)f$ for every $t \in V$, i.e., $f \in \mathcal C_{\eta}$. 

The map $\Phi_\eta$ is clearly linear. Define $\Psi_\eta : \mathcal C_\eta \to \mathbb C^{\Gr{K_\eta}{k}}$ by $(\Psi_\eta f)(U)= f(0 + U)$. Then, it is straightforward to verify that $\Phi_\eta \Psi_\eta f = f$ for $f \in \mathcal C_\eta$. Thus, $\Psi_\eta = \Phi_\eta^{-1}$ and $\Phi_\eta$ is a linear isomorphism. Finally, the decomposition follows from Lemma \ref{lem:character decomposition}. 
\end{proof}

We now establish the Bey-type inequality for affine space as below.

\begin{lemma}\label{lem:affine-bey}
Assume $n\ge2k+1$. Then for every $\mathcal A\subseteq\mathcal M_k$ with $m=|\mathcal A|$, the following inequality holds:
\[
    \co^{\mathrm{AG}}(\mathcal A)
\le q[k]_q[n-k]_qm+\frac{q[k]_q}{\qbinom{n}{k}}m^2.
\]
\end{lemma}

\begin{proof}
Take $\mathscr{X}=\mathcal M_k$ and $\mathscr{Y}=\mathcal M_{k-1}$, and let $B$ and $T=BB^T$ be as defined before. For any $P,Q\in \mathcal{M}_k$, we have
\[
T_{P,Q}=\sum\limits_{E\in \mathcal M_{k-1}}B_{P,E}B_{Q,E}=|\{E\in\mathcal M_{k-1}: E\subseteq P\cap Q\}|.
\]
Thus, $T_{P,Q}$ is the number of $(k-1)$-flats contained in both $P$ and $Q$. In particular, by Lemma \ref{lem:affine-counts}, $T_{P,P}=q[k]_q$. Moreover, for every $P\in\mathcal M_k$,

\[
  \sum_{Q \in \mathcal M_k} T_{P,Q} 
= \sum_{\substack{E \in \mathcal M_{k-1}, E \subseteq P}} 
  |\{Q \in \mathcal M_k : E \subseteq Q\}|
= q[k]_q[n - k + 1]_q.
\]
Thus, let $\mathbf{1}\in \mathbb{C}^{\mathcal{M}_k}$ denote the constant-one function, that is, $\mathbf{1}(P)=1$ for every $P\in\mathcal{M}_k$. Under the natural identification $f \mapsto (f(P))_{P \in \mathcal{M}_k}$, the function $\mathbf{1}$ is the all-one vector indexed by $\mathcal M_k$, then $T\mathbf 1=q[k]_q[n-k+1]_q\mathbf 1$, i.e., $\lambda_0=q[k]_q[n-k+1]_q$. 

Let $\mathbb{C}\mathbf{1} := \{c\mathbf{1} : c \in \mathbb{C}\}$ be the one-subspace of constant functions. With respect to the standard Hermitian inner product 
\[\langle f, g \rangle_{\mathcal{M}_k} := \sum_{P \in \mathcal{M}_k} f(P) \overline{g(P)},\]
its orthogonal complement is  
\[(\mathbb{C}\mathbf{1})^\perp := \{f \in \mathbb{C}^{\mathcal{M}_k} : \langle f, \mathbf{1} \rangle_{\mathcal{M}_k} = 0\}=\bigg\{ f \in \mathbb{C}^{\mathcal{M}_k} : \sum_{P \in \mathcal{M}_k} f(P) = 0 \bigg\}.\]
Finally, for every $f\in(\mathbb{C}\mathbf{1})^\perp$, the symmetry of $T$ gives
\[
    \langle Tf,\mathbf{1}\rangle_{\mathcal{M}_k}=\langle f,T\mathbf{1}\rangle_{\mathcal{M}_k}=\lambda_0\langle f,\mathbf{1}\rangle_{\mathcal{M}_k}=0.
\]
Thus, $(\mathbb{C}\mathbf{1})^\perp$ is $T$-invariant. It remains to determine $\lambda_1 := \lambda_{\text{max}} \left( T |_{(\mathbb{C}\mathbf{1})^\perp} \right)$

We regard $V$ as its additive group $(V,+)$, which acts on $\mathcal{M}_k$ by translation $(a+U)\cdot t=(a+t)+U$ for $a,t \in V$. Then the translation preserves containment. In particular, for every $P, Q \in\mathcal{M}_k$ and $t\in V$, the map $E \mapsto E + t$ is a bijection between the affine $(k-1)$-flats contained in $P \cap Q$ and those contained in $(P + t) \cap (Q + t)$. Thus 
\begin{equation}\label{eq: translation}
    T_{P+t, Q+t} = T_{P, Q}.
\end{equation} 
For $t\in V$, let $\rho_t: \mathbb C^{\mathcal M_k} \to \mathbb C^{\mathcal M_k}$ be defined by $(\rho_t f)(P) := f(P + t)$. Then \eqref{eq: translation} implies $T \rho_t = \rho_t T$ for $t\in V$. Let $K_0=V$. Lemma \ref{lem:character decomposition} therefore implies that every $\mathcal{C}_\eta$ is $T$-invariant. By Lemma \ref{lem:affine-character-components}, the function space admits the orthogonal decomposition 
\[
\mathbb C^{\mathcal M_k}
=
\bigoplus_{\eta\in V^*}^{\perp}\mathcal C_\eta.
\]
Fix $\eta \in V^*$, and let $K_\eta=\ker \eta$; in particular, $K_0 = V$. By Lemma \ref{lem:affine-character-components}, every  
$f \in \mathcal C_\eta$ has a unique representation $f = \Phi_\eta g$ for $g \in \mathbb C^{\Gr{K_\eta}{k}}$, where the linear isomorphism
$\Phi_\eta: \mathbb{C}^{\Gr{K_\eta}{k}}\to \mathcal{C}_\eta$ is defined by
\[
(\Phi_\eta g)(a + U) = 
\begin{cases} 
\psi(\eta(a))g(U), & U \subseteq K_\eta, \\
0, & U \nsubseteq K_\eta.
\end{cases}
\]

We now compute $T\Phi_\eta g$. Since $\mathcal C_{\eta}$ is $T$-invariant, $T\Phi_{\eta}g \in \mathcal C_{\eta}$. Therefore, we obtain that $(T\Phi_{\eta}g)(a + U)=0$ whenever $U \not\subseteq K_{\eta}$ by Lemma \ref{lem:affine-character-components}, It remains to compute $(T\Phi_{\eta}g)(a + U)$ for $U \subseteq K_{\eta}$. Let $P=a+U$ with $U\subseteq K_\eta$. A $(k-1)$-flat contained in $a + U$ has the form 
\[a + u + H, \quad H \subseteq U, \quad \dim(H)= k-1, \quad u + H \in U/H.\]
For a fixed $H$, the quotient $U/H$ has $q$ elements. The $k$-flats containing $a + u + H$ on which $\Phi_\eta g$ may be nonzero are precisely $a+u+L$, where $H\subseteq L\subseteq K_\eta$ and $\dim(L) = k$. Since $u \in U \subseteq K_\eta$, $\eta(a + u) = \eta(a)$. Therefore,
\[
\begin{aligned}
    (T\Phi_\eta g)(a + U)&=\sum_{\substack{H\subseteq U \\ \dim(H) = k-1}} \sum_{u+H\in U/H} \sum_{\substack{H\subseteq L\subseteq K_\eta \\ \dim(L) = k}} \psi(\eta(a+u)) g(L)\\
    &=q \psi(\eta(a)) \sum_{\substack{H\subseteq U \\ \dim(H) = k-1}} \sum_{\substack{H\subseteq L\subseteq K_\eta \\ \dim(L)= k}} g(L).
\end{aligned}
\]

Let $C_{K_\eta}$ be the incidence matrix between the $k$-subspaces and the $(k-1)$-subspaces of $K_\eta$. Then 
\[(C_{K_\eta} C_{K_\eta}^T g)(U) = \sum_{\substack{H\subseteq U \\ \dim (H)= k-1}} \sum_{\substack{H\subseteq L\subseteq K_\eta \\ \dim(L) = k}} g(L).\]
Therefore, $T\Phi_\eta=\Phi_\eta(q C_{K_\eta}C_{K_\eta}^T)$. Equivalently, $\Phi_\eta^{-1} \left( T |_{C_\eta} \right) \Phi_\eta = q C_{K_\eta} C_{K_\eta}^T$. Let $\mathbf{1}$ denote the all-one function on $\mathcal{M}_k$, and set $\lambda_1= \lambda_{\max} \left( T |_{(\mathbb C \mathbf{1})^\perp} \right)$. Since the decomposition is orthogonal and $\mathbf{1} \in \mathcal{C}_0$, every $C_\eta$ with $\eta \neq 0$ is contained in $(\mathbb{C}\mathbf{1})^\perp$. Thus  
\[(\mathbb{C}\mathbf{1})^\perp = (\mathcal{C}_0 \cap (\mathbb{C}\mathbf{1})^\perp) \oplus \bigoplus_{\eta \neq 0} C_\eta.\]
First consider $\mathcal C_0$. Since $K_0 = V$,  
$\Phi_0$ identifies $\mathcal C_0$ with the functions on  
$\Gr{V}{k}$. Moreover, if $\mathbf{1}'$ denotes the all-one function on $\Gr{V}{k}$, then $\langle \Phi_0 g, \mathbf{1} \rangle = q^{n-k} \langle g, \mathbf{1}' \rangle$ since every direction $U$ has $q^{n-k}$ affine translates. Thus, by Lemma \ref{prop:grassmann-spectrum}, we have 
\[\lambda_{\text{max}} \left( T |_{\mathcal{C}_0 \cap (\mathbb{C}\mathbf{1})^\perp} \right) = q \left( q [k-1]_q [n-k]_q \right) = q^2 [k-1]_q [n-k]_q.\]
If $\eta \neq 0$, then $\dim(K_\eta) = n-1$. By Lemma \ref{prop:grassmann-spectrum}, we have
\[\lambda_{\text{max}}(T|_{C_\eta}) = q [k]_q [(n-1) - k + 1]_q = q [k]_q [n - k]_q.\]
Therefore, 
\[\lambda_1 = \max \{ q^2[k - 1]_q[n - k]_q, q[k]_q[n - k]_q \}=q[k]_q[n-k]_q.\]
Finally, by \eqref{eq:qinteger-sum}, $\lambda_0-\lambda_1=q[k]_q\bigl([n-k+1]_q-[n-k]_q\bigr)=q[k]_qq^{n-k}$. Since $|\mathcal M_k|=q^{n-k}\qbinom{n}{k}$, Lemma \ref{lem:general-spectral} gives
\[
\operatorname{co}^{\mathrm{AG}}_2(\mathcal A)
\le q[k]_q[n-k]_qm
+q[k]_qq^{n-k}\frac{m^2}{q^{n-k}\qbinom{n}{k}}
=q[k]_q[n-k]_qm+\frac{q[k]_q}{\qbinom{n}{k}}m^2,\]
and the proof is complete.
\end{proof}

\vskip 3mm
\begin{proof}[\bfseries{Proof of Theorem \ref{thm:affine-l2-ekr}}]
Let $m=|\mathcal F|$. By Theorem \ref{thm:guo-xu}, we have $m\le M:=\qbinom{n}{k}$. Lemma \ref{lem:affine-bey} gives
\[
\operatorname{co}^{\mathrm{AG}}_2(\mathcal F)\le \mu_2(m),
\qquad
\mu_2(x)=q[k]_q[n-k]_qx+\frac{q[k]_q}{\qbinom{n}{k}}x^2.
\]
The function $\mu_2$ is strictly increasing on $[0,\infty)$, since both coefficients are nonnegative and the quadratic coefficient is positive.  Thus,
\[
\operatorname{co}^{\mathrm{AG}}_2(\mathcal F)
\le \mu_2(M)
=q[k]_q[n-k]_q\qbinom{n}{k}+q[k]_q\qbinom{n}{k},
\]
which is exactly \eqref{eq:affine-main-first}.

We now compute the value on a family consisting of all $k$-flats that contain a fixed point $p$.  Fix a $0$-flat $p$ and let $\mathcal S_p=\{P\in\mathcal M_k:p\in P\}$. Let $E\in\mathcal M_{k-1}$.  If $p\in E$, then every $k$-flat containing $E$ belongs to $\mathcal S_p$, so Lemma \ref{lem:affine-counts} gives $d_{\mathcal S_p}(E)=[n-k+1]_q$. The number of $(k-1)$-flats through $p$ is $\qbinom{n}{k-1}$, because such flats are determined by their directions.

If $p\notin E$, let $E=b+H$ with $\dim(H)=k-1$. Since $p-b\notin H$, $A=\aff(E\cup\{p\})=b+(H+\Span\{p-b\})$ is a $k$-flat containing $E$ and $p$. If $A_1=p+H_1$ is any $k$-flat containing both $E$ and $p$, then $A\subseteq A_1$ by the definition of the affine span. Since the directions of $A$ and $A_1$ have the same dimension, it follows that $A=A_1$. Thus $A$ is the unique $k$-flat containing $E$ and $p$, and therefore $d_{\mathcal S_p}(E)=1$. The total number of $(k-1)$-flats is $q^{n-k+1}\qbinom{n}{k-1}$, so the number not containing $p$ is $(q^{n-k+1}-1)\qbinom{n}{k-1}$. Thus
\begin{equation}\label{eq:affine-pencil-codegree}
\operatorname{co}^{\mathrm{AG}}_2(\mathcal S_p)
=\qbinom{n}{k-1}[n-k+1]_q^2
+(q^{n-k+1}-1)\qbinom{n}{k-1},
\end{equation}
which is the right-hand side of \eqref{eq:affine-main-second}.

It remains to see that \eqref{eq:affine-main-second} agrees with \eqref{eq:affine-main-first}. Applying \eqref{eq:qbinom-ratio} and the elementary equality $[k]_q=1+q[k-1]_q$, thus
\[
\begin{aligned}
\qbinom{n}{k-1}\bigl([n-k+1]_q^2+q^{n-k+1}-1\bigr)
&=\qbinom{n}{k-1}[n-k+1]_q\bigl([n-k+1]_q+q-1\bigr)\\
&=\qbinom{n}{k}[k]_q\,q\bigl([n-k]_q+1\bigr),
\end{aligned}
\]
which is \eqref{eq:affine-main-first}.  Thus $\mathcal S_p$ attains the upper bound.

If equality holds in \eqref{eq:affine-main-first}, then the strict monotonicity of $\mu_2$ forces $m=M=\qbinom{n}{k}$.  The equality statement in Theorem \ref{thm:guo-xu} then forces $\mathcal F=\mathcal S_p$ for some $0$-flat $p$.  Conversely, every $\mathcal S_p$ attains equality by \eqref{eq:affine-pencil-codegree}.
\end{proof}

\section{Attenuated spaces}
Similar as before, we first establish the Bey-type inequality for attenuated space, and then give the proof of Theorem \ref{thm:attenuated-l2-ekr}.


The following lemma is used to establish the Bey-type inequality.
\begin{lemma}\label{lem: trace pairing}
Let $N$ and $W$ be finite-dimensional vector spaces over $\mathbb{F}_q$. The bilinear pairing
\[
\beta : \Hom(W, N) \times \Hom(N, W) \to \mathbb{F}_q, \quad \beta(D, A) := \Tr(DA),
\]
is perfect. Here $\Tr$ denotes the trace of an endomorphism of $N$, which is independent of the choice of basis of $N$. Consequently, the functions
\[
\chi_D(A) := \psi(\Tr(DA)), \quad D \in \Hom(W, N), 
\]
are precisely the characters of the additive group $G = \Hom(N, W)$,
and they form an orthogonal basis of $\mathbb{C}^G$.    
\end{lemma}
\begin{proof}
Suppose that $A\in\Hom(N,W)$ is nonzero. Choose $x \in N$ such that $A(x) \neq 0$, and choose $\alpha \in W^*$ such that $\alpha(A(x)) = 1$. Define $D\in\Hom(W,N)$ by $D(w)=\alpha(w)x$. Then $DA(v) = \alpha(A(v))x$. The trace of the rank-one operator $v \mapsto \gamma(v)x$ is $\gamma(x)$; thus $\Tr(DA) = \alpha(A(x)) = 1$. Therefore the pairing is nondegenerate in the second variable. Since $\Hom(W, N)$ and $\Hom(N, W)$ have the same finite dimension $\dim(N)\dim(W)$, the pairing is perfect. The conclusion now follows from Lemma \ref{lem:perfect pairing}. 
\end{proof}

We now  derive the corresponding Bey-type inequality for attenuated space.

\begin{lemma}\label{lem:attenuated-bey}
For every $\mathcal A\subseteq\mathcal M_n$, with $m=|\mathcal A|$, one has
\[
\operatorname{co}^{\mathrm{AT}}_2(\mathcal A)
\le q^\ell[n-1]_qm+q^{n-1-\ell(n-1)}m^2.
\]
\end{lemma}

\begin{proof}
Let $\mathscr{X}=\mathcal M_n$ and $\mathscr{Y}=\mathcal M_{n-1}$, and let $B$ and $T=BB^T$ be as defined before. As stated in Section 1, we identify $\mathcal{M}_n$ as $G=\Hom(N,W)$ by $A\leftrightarrow P_A$. Since the justification for this equality was deferred, we provide a complete derivation below. In fact, let $V=N\oplus W$ with $\dim(N)=n$. For any $P \in \mathcal{M}_n$, we have $P \cap W = \{\bm{0}\}$. Consider the linear projection

\[
\pi : V \to N, \quad \pi(x + y) = x \quad (x \in N, \, y \in W).
\]
Its restriction $\pi|_P : P \to N$ is injective, since $p \in \ker(\pi|_P)$ implies $p \in P \cap W = \{\bm{0}\}$. As $\dim(P)=\dim(N)=n$, this restriction is an isomorphism. Thus for every $x \in N$, there is a unique $p_x \in P$ with $\pi(p_x)=x$, and we may write uniquely $p_x = x + A(x)$ for some $A(x) \in W$. The map $x \mapsto A(x)$ is linear by construction, so $A \in \Hom(N,W)$. Conversely, every $A \in \Hom(N,W)$ defines an $n$-subspace $P_A=\{x + A(x): x \in N\}$, which clearly satisfies $P_A \cap W =\{\bm{0}\}$. Thus the correspondence $A \leftrightarrow P_A$ is a bijection between $\Hom(N,W)$ and $\mathcal M_n$. 

For $P_A,P_B\in\mathcal{M}_n$, the entry of $T$ indexed by $P_A$ and $P_B$, denoted by $T_{A,B}$, is given by 
\[
   T_{A,B}=|\{E\in\mathcal{M}_{n-1}:\ E\subseteq P_A\cap P_B\}|. 
\]
Moreover, $P_A\cap P_B=\{x + A(x): x \in \ker(A-B)\}$, thus $P_A\cap P_B$ is naturally isomorphic to $\ker(A-B)$, and $\dim(P_A\cap P_B)=n-\rank(A-B)$, where $\rank(A-B)$ is the dimension of the image space of $A-B$.

For $C\in G$, define $\Phi_C:\ V\to V$, $\Phi_C(x+w)=x+w+C(x)$ for $x\in N, w\in W$. Then $\Phi_C$ is a linear automorphism with inverse $\Phi_{-C}$, and $\Phi_C(W)=W$, $\Phi_C(P_A)=P_{A+C}$. Moreover, if $E_{H,a}=\{x+a(x): x\in H\}\in\mathcal{M}_{n-1}$, then $\Phi_C(E_{H,a})=E_{H,a+C|_H}$. Set $E+C:=\Phi_C(E)$. Thus, $\Phi_C$ is a bijection from $\mathcal{M}_{n-1}$ to itself and preserves containment. 
Therefore, the mapping $E\mapsto E+C$ is a bijection between the $(k-1)$-subspaces contained in $P_A\cap P_B$ and those contained in $P_{A+C}\cap P_{B+C}$. Therefore, $T_{A+C,B+C}=T_{A,B}$. Taking $C=-B$, we get $T_{A,B}=T_{A-B,0}$, so $T_{A,B}$ depends only on $A-B$. Thus there exists a function $h: G \to \mathbb{Z}_{\geq0}$ such that $T_{A,B} = h(A-B)$. Therefore, for $f\in\mathbb{C}^G$
\[
(Tf)(A) = \sum_{B \in G} T_{A,B} f(B) = \sum_{B \in G} h(A-B) f(B),
\]
which is exactly convolution on $G$. By Lemma \ref{lem: trace pairing}, the functions $\chi_D(A)=\psi(\Tr(DA))$ with $D\in\Hom(W,N)$ are precisely the characters of $G$ and form an orthogonal basis of complex-valued functions on $G$. By Lemma \ref{lem: convolution}, each $\chi_D$ is an eigenvector, with eigenvalue $\lambda(D)=\sum_{B\in G}T_{0,B}\chi_D(B)$. We now compute $\lambda(D)$. Let $0$ denote the zero map, so $P_0=N$. Every member of $\mathcal{M}_{n-1}$ has the form $E_{H,a}=\{x+a(x):\ x\in H\}$ where $H\subseteq N$ is a hyperplane and $a: H\to N$ is linear. A common $(n-1)$-subspace of $P_0$ and $P_B$ is exactly a hyperplane $H\subseteq N$ such that $B|_H=0$, which implies that 
\[
    T_{0,B}=|\{H\subseteq N:\ \dim(H)=n-1, B|_H=0\}|.
\]
Therefore the eigenvalue of $\chi_D$ is
\[
\lambda(D)=\sum_{B\in\mathcal{M}_n}T_{0,B}\chi_D(B)
=\sum_{\substack{H\subseteq N\\ \dim (H)=n-1}}
  \sum_{\substack{B\in\Hom(N,W)\\ B|_H=0}}\chi_D(B).
\]
Fix such a hyperplane $H$, let $L_H=\{B\in\Hom(N,W):B|_H=0\}$. Since $B\in L_H$ factors through the one-dimensional quotient $N/H$, $L_H\cong \Hom(N/H,W)$, which implies $\dim(L_H)=\ell$ and $|L_H|=q^\ell$. By Lemma \ref{lem: character sum}, 

\[
\sum_{\substack{B\in L_H}}\chi_D(B)=
\begin{cases} 
q^\ell, & \chi_D(B)|_{L_H} \text{ is trivial,} \\
0, & \chi_D(B)|_{L_H} \text{ is nontrivial.}
\end{cases}
\]

We determine when $\chi_D(B)|_{L_H}$ is trivial. Choose a nonzero functional $\alpha:N\to\Fq$ with $\ker\alpha=H$.  Every $B_w\in L_H$ is of the form $B_w(x)=\alpha(x)w$ for a unique $w\in W$. Then $DB_w:N\to N$ is the rank-at-most-one  operator $x\longmapsto \alpha(x)D(w)$. The trace of the rank-at-most-one operator $x\mapsto\alpha(x)v$ is $\alpha(v)$; applying this with $v=D(w)$ gives $\Tr(DB_w)=\alpha(D(w))$. Thus $\chi_D$ is trivial on $L_H$ if and only if $\alpha(D(w))=0$ for every $w\in W$, which is equivalent to $D(W)\subseteq H$. If $\rank(D)=j$, then $D(W)$ is a fixed $j$-dimensional subspace of $N$, and the number of hyperplanes of $N$ containing it is $[n-j]_q$.  Therefore, the eigenvalue depends only on $j$ and is
\[
\lambda_j=q^\ell[n-j]_q, \qquad j=0,1,\ldots,n.
\]
For $D=0$, the character $\chi_D$ is the all-one function, and $\lambda_0=q^\ell[n]_q$. If $D\neq0$, then $\rank(D)\geq1$. Since $[n-j]_q$ decreases with $j$, the largest eigenvalue on the
orthogonal complement of the all-one function occurs at $j=1$ and equals $\lambda_1=q^\ell[n-1]_q$. 
Moreover, $\lambda_0-\lambda_1=q^\ell([n]_q-[n-1]_q)=q^{\ell+n-1}$. Since $|X|=|G|=q^{\ell n}$, Lemma \ref{lem:general-spectral} gives
\[
\operatorname{co}^{\mathrm{AT}}_2(\mathcal A)
\le q^\ell[n-1]_qm+q^{\ell+n-1}\frac{m^2}{q^{\ell n}}
=q^\ell[n-1]_qm+q^{n-1-\ell(n-1)}m^2, 
\]
and the proof is complete.
\end{proof}

\vskip 3mm
\begin{proof}[\bfseries{Proof of Theorem \ref{thm:attenuated-l2-ekr}}]
Let $m=|\mathcal F|$.  By Theorem \ref{thm:Huang}, we have $m\le M:=q^{\ell(n-1)}$. Lemma \ref{lem:attenuated-bey} gives
\[
\operatorname{co}^{\mathrm{AT}}_2(\mathcal F)\le \mu_3(m),
\qquad
\mu_3(x)=q^\ell[n-1]_qx+q^{n-1-\ell(n-1)}x^2.
\]
The function $\mu_3$ is strictly increasing on $[0,\infty)$.  Hence
\[
\operatorname{co}^{\mathrm{AT}}_2(\mathcal F)
\le \mu_3(M)
=q^\ell[n-1]_qq^{\ell(n-1)}
+q^{n-1-\ell(n-1)}q^{2\ell(n-1)},
\]
which simplifies to \eqref{eq:attenuated-main}.

We next compute the value on the subfamily of $\mathcal M_n$ consisting of all members that contain a fixed $1$-subspace $L\subseteq V$ satisfying $L\cap W=\{\bm{0}\}$. Fix a $1$-subspace $L\subseteq V$ with $L\cap W=\{\bm{0}\}$ and let $\mathcal S_L=\{P_A:A|_L=0\}$. Let $E_{H,a}$ be a $(n-1)$-subspace, where $H\subseteq N$ is a hyperplane and $a:H\to W$ is linear.

If $L\subseteq H$ and $a|_L=0$, then every extension $A:N\to W$ of $a$ automatically satisfies $A|_L=0$.  There are $q^\ell$ such extensions, so $d_{\mathcal S_L}(E_{H,a})=q^\ell$. There are $[n-1]_q$ hyperplanes $H$ containing $L$, and for each such $H$ the number of maps $a:H\to W$ with $a|_L=0$ is $q^{\ell\dim(H/L)}=q^{\ell(n-2)}$. 
Their total contribution to the squared codegree sum is $[n-1]_q q^{\ell(n-2)}(q^\ell)^2=[n-1]_q q^{\ell n}$.

If $H$ does not contain $L$, then $N=H\oplus L$.  For every map $a:H\to W$, there is a unique extension $A:N\to W$ such that $A|_H=a$ and $A|_L=0$.  Therefore, $d_{\mathcal S_L}(E_{H,a})=1$. 
The number of hyperplanes of $N$ not containing $L$ is $[n]_q-[n-1]_q=q^{n-1}$ and for each such $H$, there are $q^{\ell(n-1)}$ choices for $a$. This case contributes $q^{n-1}q^{\ell(n-1)}$.

If $L\subseteq H$ but $a|_L\ne0$, then no extension $A$ with $A|_L=0$ can contain $E_{H,a}$, so the contribution is zero. Combining the cases,
\[
\operatorname{co}^{\mathrm{AT}}_2(\mathcal S_L)
=[n-1]_qq^{\ell n}+q^{n-1}q^{\ell(n-1)}
=q^{\ell(n-1)}\bigl(q^\ell[n-1]_q+q^{n-1}\bigr).
\]
Thus $\mathcal S_L$ attains the bound for a fixed $1$-subspace $L$ with $L\cap W=\{\bm{0}\}$.

Now assume $\ell=n$ and consider a dual family
\[
\mathcal D_U=\{P\in\mathcal M_n:P\subseteq U\},
\quad
\dim (U)=2n-1,
\quad
S=U\cap W,
\quad
\dim (S)=n-1.
\]
If $E\not\subseteq U$, then $d_{\mathcal D_U}(E)=0$. Suppose $E\subseteq U$.  Since $E\in\mathcal M_{n-1}$, it is disjoint from $W$ and hence from $S\subseteq W$.  A member $P\in\mathcal D_U$ containing $E$ is obtained by choosing a one-dimensional subspace of $U/E$ and adding it to $E$.  However, the chosen line must not lie in $(E+S)/E$, because then the resulting $P$ would meet $S\subseteq W$ nontrivially and hence would not belong to $\mathcal M_n$.  Since $\dim(U/E)=n$ and $\dim((E+S)/E)=n-1$, the number of allowed lines is $[n]_q-[n-1]_q=q^{n-1}$. Thus $d_{\mathcal D_U}(E)=q^{n-1}$ for each admissible $E\subseteq U$.

It remains to count such $E$.  Choose a complement $T$ of $S$ in $U$, so $\dim (T)=n$.  An $(n-1)$-subspace $E\subseteq U$ with $E\cap S=\{\bm{0}\}$ is determined by first choosing an $(n-1)$-subspace $H\subseteq T$ and then choosing a linear map $a:H\to S$ such that $E_{H,a}=\{x+a(x):\ x\in H\}$.  Hence the number of such $E$ is $\qbinom{n}{n-1}q^{(n-1)^2}=[n]_q q^{(n-1)^2}$.
Therefore
\[
\operatorname{co}^{\mathrm{AT}}_2(\mathcal D_U)
=[n]_q q^{(n-1)^2}(q^{n-1})^2
=[n]_q q^{(n-1)(n+1)}.
\]
For $\ell=n$, the upper bound \eqref{eq:attenuated-main} equals
\[
q^{n(n-1)}\bigl(q^n[n-1]_q+q^{n-1}\bigr)
=q^{n(n-1)}q^{n-1}\bigl(q[n-1]_q+1\bigr)
=[n]_q q^{(n-1)(n+1)},
\]
because $[n]_q=q[n-1]_q+1$.  Hence dual families also attain the bound when $\ell=n$.

Finally, if equality holds in \eqref{eq:attenuated-main}, then the strict monotonicity of $\mu_3$ forces $m=M=q^{\ell(n-1)}$.  The equality statement in  Theorem \ref{thm:Huang}, identifies $\mathcal F$ as one of the listed extremal families.  Conversely, the computations above show that all listed extremal families attain equality.
\end{proof}

\section*{Acknowledgement}
\noindent This research is supported by National Key R\&D Program of China under grant number 2024YFA1013900 and NSFC under grant number 12471327.

\end{document}